\documentclass[12pt]{amsart}
\usepackage{etoolbox}

\makeatother

\usepackage[utf8]{inputenc}
\usepackage{graphicx}
\graphicspath{ {figures/} }

\usepackage{latexsym}
\usepackage{amsmath}
\usepackage{amsfonts}
\usepackage{amssymb}
\usepackage{enumerate}
\usepackage{relsize}
\usepackage{color}
\usepackage{graphicx}
\usepackage{float}
\usepackage[margin=0.8in]{geometry}
\usepackage[dvipsnames]{xcolor}
\usepackage{outlines}
\usepackage{soul}
\usepackage{comment}
\usepackage{tabularx}
\usepackage{caption}
\usepackage{blindtext}
\usepackage[all]{xy}
\usepackage{tikz-cd}
\usetikzlibrary{matrix,arrows} 
\usepackage{mathtools}
\usepackage{adjustbox} 
\usepackage{mwe} 
\usepackage{mathtools} 
\usepackage{accents} 
\usepackage{tikz-cd}
\usepackage{setspace}
\usepackage{amscd}

\usepackage[colorlinks = true,
            linkcolor = RoyalBlue,
            urlcolor  = blue,
            citecolor = blue,
            anchorcolor = blue]{hyperref}

\newtheorem{theorem}[equation]{Theorem}

\newtheorem{proposition}[equation]{Proposition}

\newtheorem{conjecture}[equation]{Conjecture}
\newtheorem{definition-lemma}[equation]{Definition-Lemma}
\theoremstyle{definition}
\newtheorem{definition}[equation]{Definition}

\newtheorem{example}[equation]{Example}

\theoremstyle{remark}
\newtheorem{remark}[equation]{Remark}

\numberwithin{equation}{section}
\numberwithin{figure}{section}

\newcommand {\Aut} {\operatorname{\text{Aut}}}

\newcommand {\Curv} {\operatorname{\text{\ovA{Curv}}}}
\newcommand {\cont} {\operatorname{\text{cont}}}

\newcommand {\Eff}  {\operatorname{Eff}}
\newcommand {\Fix}  {\operatorname{Fix}}
\newcommand {\Gal} {\operatorname{\text{Gal}}}

\newcommand {\Mov} {\operatorname{Mov}}
\newcommand {\Move} {\operatorname{\Mov^{\; e}}}
\newcommand {\Nef} {\operatorname{Nef}}
\newcommand {\Nefe} {\operatorname{\Nef^{\; e}}}
\newcommand {\Pic}  {\operatorname{Pic}}

\newcommand {\PsAut}  {\operatorname{PsAut}}

\newcommand\widebar[1]{\mathop{\overline{#1}}}

\makeatletter
\newcommand*{\ov}[1]{%
  $\m@th\overline{\mbox{#1}}$%
}
\newcommand*{\ovA}[1]{%
  $\m@th\overline{\mbox{#1}\raisebox{3mm}{}}$%
}
\newcommand*{\ovB}[1]{%
  $\m@th\overline{\mbox{#1\rule{0pt}{3mm}}}$%
}
\newcommand*{\ovC}[1]{%
  $\m@th\overline{\mbox{#1\strut}}$%
}
\newcommand*{\ovD}[1]{%
  $\m@th\overline{\mbox{#1\vphantom{\"A}}}$%
}
\newcommand*{\ovE}[1]{%
  $\m@th\overline{\raisebox{0pt}[1.2\height]{#1}}$%
}
\newcommand*{\ovF}[1]{%
  $\m@th\overline{\raisebox{0pt}[\dimexpr\height+1mm\relax]{#1}}$%
}
\newcommand*{\ovG}[1]{%
  $\m@th\overline{\raisebox{0pt}[\dimexpr\height+1mm\relax]{#1\vphantom{A}}}$%
}
\makeatother

\begin{document}

\title{On the cone conjecture for log Calabi-Yau mirrors of Fano 3-folds}
\author{Jennifer Li}
\maketitle

\begin{abstract}
Let $Y$ be a smooth projective $3$-fold admitting a K3 fibration $f : Y \rightarrow \mathbb{P}^1$ with $-K_Y = f^*\mathcal{O}(1)$. We show that the pseudoautomorphism group of $Y$ acts with finitely many orbits on the codimension one faces of the movable cone if  $H^3(Y,\mathbb{C})=0$, confirming a special case of the Kawamata--Morrison--Totaro cone conjecture. In \cite{CCGK16} \cite{P18}, and \cite{CP18}, the authors construct log Calabi-Yau 3-folds with K3 fibrations satisfying the hypotheses of our theorem as the mirrors of Fano 3-folds. 
\end{abstract}

%
%

\section{Introduction}
\label{sec:Introduction}

Morrison's cone conjecture is stated as follows:
\begin{conjecture}(The Morrison cone conjecture)
Let $X$ be a Calabi-Yau 3-fold. Then,
\begin{enumerate}
\item The automorphism group of $X$ acts on the nef cone of $X$ with a rational polyhedral fundamental domain; and
\item The pseudoautomorphism group of $X$ (see Definition \ref{def:pseudoautomorphism} below) acts on the movable cone of $X$ with a rational polyhedral fundamental domain.
\end{enumerate}
\end{conjecture}

In particular, Morrison's conjecture implies the following:
\begin{enumerate}
\item The automorphism group of $X$ acts on the faces of the nef cone of $X$ with finitely many orbits; and
\item The pseudoautomorphism group of $X$ acts on the faces of the movable cone of $X$ with finitely many orbits.
\end{enumerate}

The Morrison cone conjecture was generalized by Totaro to klt pairs (see \cite{T10}). In this paper, we prove a version of the cone conjecture for certain types of log Calabi-Yau 3-folds, namely, for smooth projective 3-folds $Y$ that admit a {\it K}3 fibration $f: Y \rightarrow \mathbb{P}^{1}$ such that $-K_{Y} = f^{\ast} \mathcal{O}(1)$.

\begin{remark}
If $f : Y \rightarrow \mathbb{P}^1$ is a K3 fibration such that $K_Y$ is relatively nef, $-K_Y$ is effective and nonzero, $f$ is not isotrivial, and $f$ has no multiple fibers, then $-K_Y=f^*\mathcal{O}(1)$. This follows from positivity of the pushforward of the relative canonical bundle, see e.g. \cite{F78}.
\end{remark}

\begin{remark}
If $D$ is any smooth (or reduced normal crossing) fiber of $f$, then $(Y, D)$ is log Calabi-Yau.
\end{remark}

\begin{definition}
The {\em cone of curves} $\Curv(Y)$ is defined as:
\begin{center}
$\Curv(Y) = \widebar{\{ \sum a_{i} [C_{i}] \; \vert \; a_{i} \in \mathbb{R}_{\geq 0} \text{ and } C_{i} \subset Y \text{ is a curve } \}}$.
\end{center}
\end{definition}

\begin{definition}
The {\em nef cone} $\Nef(Y)$ is the dual cone of the cone of curves. The {\em nef effective cone} is defined as $\Nefe(Y) := \Nef(Y) \cap \Eff(Y)$.
\end{definition}

\begin{definition}
The {\em movable cone} is defined as:
\begin{center}
$\Mov(Y) := \widebar{\langle B \; \vert \; B \text{ is a divisor with } \vert B \vert \text{ having no fixed part} \rangle}_{\mathbb{R} \geq 0}$.
\end{center}
The {\em movable effective cone} is defined as $\Move(Y) := \Mov(Y) \cap \Eff(Y)$.
\end{definition}

\begin{definition}
\label{def:pseudoautomorphism}
A pseudoautomorphism of a variety $Y$ is a birational map $\theta: Y \dashrightarrow Y$ that is an isomorphism outside of codimension 2 subsets of the domain and codomain.
\end{definition}

\begin{definition}
A {\it K}3 fibration is a morphism $f: Y \rightarrow B$ such that the general fiber $F \subset Y$ of $f$ is a {\it K}3 surface.
\end{definition}

Here we briefly lay out some key results used in this paper. First of all, we note that our setup is a special case of the Kawamata-Morrison-Totaro cone conjecture (see Theorem \ref{thm:KMTconeConjecture}). This allows us to conclude that $\Mov(Y)$ is the union of the nef cones of all flops of $Y$ (see Theorem \ref{thm:BCHMMovableCone}). Two important results that we will use next are Mori's cone theorem (\cite{KM98}, Theorem 1.24) and Mori's classification of extremal rays of the cone of curves for smooth 3-folds (\cite{M82}). Given a smooth projective variety $Y$, Mori's cone theorem describes the cone of curves of $Y$ in the region $K_{Y} < 0$. Mori gives a classification of the contractions $\mathrm{cont}_{R}: Y \rightarrow X$ associated to extremal rays $R \subset \Curv(Y)$ in the region $K_{Y} < 0$ when $Y$ is a smooth projective 3-fold over $\mathbb{C}$. There are a total of eight possibilities (see \cite{M82} or \cite{KM98}).

Now suppose that $Y$ is a smooth projective 3-fold and $f: Y \rightarrow \mathbb{P}^{1}$ is a {\it K}3 fibration such that $-K_{Y} = f^{\ast} \mathcal{O}(1)$. In this setting, we show (Proposition \ref{prop:eliminate6cases}) that for any $Z$ obtained by a sequence of flops of $Y$, six of the eight possible types of extremal rays of $\Curv(Z)$ in the region $K_{Z} < 0$ can be eliminated, leaving two remaining types:
\begin{enumerate}
\item Type (1): the blowup of a smooth curve $\Gamma$; or
\item Type (6): a conic bundle.
\end{enumerate}

We prove the following main results.

\begin{theorem}
\label{thm:mainTheorem} 
Let $Y$ be a smooth projective 3-fold and $f: Y \rightarrow \mathbb{P}^{1}$ a {\it K}3 fibration such that $-K_{Y} = f^{\ast} \mathcal{O}(1)$. Then the pseudoautomorphism group $\PsAut(Y)$ of $Y$ acts on
\begin{enumerate}
\item the Type (6) faces of $\Move(Y)$ with finitely many orbits;
\item the Type (1) faces of $\Move(Y)$ with finitely many orbits if $H^{3}(Y, \mathbb{C}) = 0$.
\end{enumerate}
\end{theorem}

\begin{remark}
If $H^{3}(Y, \mathbb{C}) = 0$, then any contraction of Type (1) has genus $g=0$.
\end{remark}

By a result of Kawamata, we can also take care of all curves on $K^{\bot}$:

\begin{theorem} (Kawamata \cite{K97})
\label{thm:Kawamata}
$\PsAut(Y)$ acts on the faces of $\Move(Y)$ containing $-K_{Y}$ with finitely many orbits.
\end{theorem}

\begin{remark}
Because $-K_{Y}$ (a fiber) is nef, there are no curves in the region $K > 0$.
\end{remark}

\begin{theorem}
\label{thm:PsAutActsFiniteOrbitsIfH3zero}
$\PsAut(Y)$ acts on the codimension 1 faces of $\Move(Y)$ with finitely many orbits if $H^{3}(Y, \mathbb{C}) = 0$.
\end{theorem}

\begin{conjecture}
\label{conj:codimOneFaces}
$\PsAut(Y)$ acts with finitely many orbits on the faces of $\Move(Y)$ of Type (1) for any genus $g \geq 0$.
\end{conjecture}

\begin{theorem}
\label{thm:logCY3orbits}
Assume that Conjecture \ref{conj:codimOneFaces} holds. Then $\PsAut(Y)$ acts with finitely many orbits on the codimension 1 faces of $\Move(Y)$.
\end{theorem}

In Section \ref{sec:examples}, we show that there are many examples of smooth projective 3-folds $Y$ admitting a {\it K}3 fibration $f: Y \rightarrow \mathbb{P}^{1}$ such that $-K_{Y} = f^{\ast} \mathcal{O}(1)$ and $Y$ has infinite pseudoautomorphism group. \\

\noindent {\bf Acknowledgements.} I thank Paul Hacking, J\'{anos} Koll\'{a}r, and Chenyang Xu for very helpful discussions and suggestions. \\

%
%

\section{Proof of main theorem}
\label{sec:proofOfMainTheorem}

Here we prove Theorem \ref{thm:mainTheorem}. Let $Y$ be a smooth projective 3-fold admitting a {\it K}3 fibration $f: Y \rightarrow \mathbb{P}^{1}$ such that $-K_{Y} = f^{\ast}\mathcal{O}(1)$. Choose
\begin{center}
$\Delta = \displaystyle{\tfrac{1}{2}}F_{1} + \displaystyle{\tfrac{1}{2}}F_{2}$,
\end{center}
where $F_{1}$ and $F_{2}$ are smooth, distinct fibers. Then the pair $(Y, \Delta)$ is klt log Calabi-Yau, since the coefficients of $\Delta$ are strictly less than one, and $K_{Y} + \Delta \sim_{\mathbb{Q}} 0$. Note that we are in the setting that Totaro considered in \cite{T10}. Thus, Theorem \ref{thm:mainTheorem} confirms the special case of the Kawamata-Morrison-Totaro cone conjecture.

\begin{theorem}
\label{thm:KMTconeConjecture}
(The Kawamata-Morrison-Totaro cone conjecture, \cite{T10}, Conjecture 2.1). Let $(Y, \Delta)$ be a klt log Calabi-Yau pair. Then
\begin{enumerate}
\item $\Aut(Y, \Delta)$ acts on $\Nefe(Y)$ with a rational polyhedral fundamental domain; and
\item $\PsAut(Y, \Delta)$ acts on $\Move(Y)$ with a rational polyhedral fundamental domain.
\end{enumerate}
\end{theorem}

Because we are in the klt log Calabi-Yau setting, we may use a result implied by Birkar-Cascini-Hacon-McKernan \cite{BCHM10}, stated as Theorem \ref{thm:BCHMMovableCone} below. This allows us to describe the movable cone of $Y$ as the union of the nef cones of $Z$ for $Y \dashrightarrow Z$ an SQM (small $\mathbb{Q}$-factorial modification).

\begin{theorem}(cf. \cite{BCHM10}, Corollary 1.1.5)
\label{thm:BCHMMovableCone}
Let $(Y, \Delta)$ be klt where $\Delta$ is an effective $\mathbb{Q}$-divisor on $Y$ and $K_{Y} + \Delta \sim_{\mathbb{Q}} 0$. Then we have a decomposition
\begin{center}
$\Move(Y) = \bigcup_{Y \dashrightarrow Z} \Nefe(Z)$
\end{center}
that is locally rational polyhedral in the interior of $\Mov(Y)$, where the union is taken over SQMs $Y \dashrightarrow Z$.
\end{theorem}

In our setting, the only type of SQM that appears is a flop (or a composition of finitely many flops). The reason is because $-K_{Y} = f^{\ast}\mathcal{O}(1)$, a fiber, and thus $-K_{Y}$ is nef, that is, $-K_{Y} \cdot C \geq 0$ for any curve $C \subset Y$. If $C \subset Y$ is an exceptional curve on an SQM, then $-K_{Y} \cdot C \geq 0$ because a smooth 3-fold has no flips (\cite{M82}). So $K_{Y} \cdot C = 0$, which corresponds to a flop. Also note that a flop of a smooth 3-fold is again smooth (\cite{K90}).

Let $S$ denote the following set:
\begin{center}
$S := \displaystyle{\bigcup_{Y \dashrightarrow Z}} \{ \text{codimension 1 faces of } \Nef(Z) \text{ in the boundary of } \Mov(Y) \}$,
\end{center}
where the union is taken over all flops $Y \dashrightarrow Z$. There is a surjection from $S$ to the collection of all codimension 1 faces of $\Mov^{e}(Y)$. Moreover, the elements of $S$ are in one-to-one correspondence with the elements of
\begin{center}
$\displaystyle{\bigcup_{Y \dashrightarrow Z}} \{ \text{extremal rays of } \Curv(Z) = \Nef(Z)^{\ast} \text{ not corresponding to small contractions} \}$.
\end{center}
The goal is to narrow down these extremal rays and the main tools used to accomplish this are the following:

\begin{enumerate}
\item Mori's cone theorem (\cite{KM98}, Theorem 1.24); and
\item Mori's classification of extremal rays of the cone of curves for smooth 3-folds (\cite{M82}).
\end{enumerate}

We state Mori's classification of extremal rays below for the reader's convenience.

\begin{remark}
Because $-K_{Z}$ is nef, there are no curves in the $K_{Z} > 0$ region. Therefore we only need to consider the $K_{Z} < 0$ region and the hyperplane $-K_{Z}^{\bot}$.
\end{remark}

\begin{theorem} 
\label{thm:MoriClassification}
(Mori's classification of extremal rays of $\Curv(Y)$, \cite{KM98}, Theorem 1.32) Let $X$ be a nonsingular projective 3-fold over $\mathbb{C}$ and $\mathrm{\cont_{R}}: Y \rightarrow X$ the contraction of a $K_{Y}$-negative extremal ray $R \subset \Curv(Y)$. Then we have the following possibilities:
\begin{enumerate}
\item $\mathrm{\cont_{R}}$ is the (inverse of the) blowup of a smooth curve in the smooth 3-fold $X$.
\item $\mathrm{\cont_{R}}$ is the (inverse of the) blowup of a smooth point of the smooth 3-fold $X$.
\item $\mathrm{\cont_{R}}$ is the (inverse of the) blowup of an ordinary double point of $X$ (locally analytically, the point is given by $x^{2} + y^{2} + z^{2} + w^{2} = 0$).
\item $\mathrm{\cont_{R}}$ is the (inverse of the) blowup of a point of $X$ that is locally analytically given by $x^{2} + y^{2} + z^{2} + w^{3} = 0$.
\item $\mathrm{\cont_{R}}$ contracts a smooth $\mathbb{P}^{2}_{\mathbb{C}}$ with normal bundle $\mathcal{O}(-2)$ to a point on $Y$ which is locally analytically the quotient of $\mathbb{C}^{3}$ by the involution $(x, y, z) \mapsto (-x, -y, -z)$.
\item $\dim(X) = 2$ and $\mathrm{\cont_{R}}$ is a fibration with plane conics as fibers (that is, $\mathrm{\cont_{R}}$ is a conic bundle).
\item $\dim(X) = 1$ and the general fibers of $\cont_{R}$ are del Pezzo surfaces.
\item $\dim(X) = 0$ and $-K_{Y}$ is ample so that $Y$ is a Fano manifold.
\end{enumerate}
\end{theorem}

\begin{proposition}
\label{prop:eliminate6cases}
Of the eight possible types of extremal rays of $\Curv(Y)$ from Theorem \ref{thm:MoriClassification}, we can 
eliminate all but (1) and (6) above: \\
\noindent {\bf Type (1).} The blowup of a smooth curve; and \\
\noindent {\bf Type (6).} A conic bundle $g: Y \rightarrow S$ (meaning that $g$ is a flat morphism and its generic fiber is a rational curve).
\end{proposition}

The main idea of the proof of the Theorem \ref{thm:mainTheorem} is as follows. We start with a smooth projective 3-fold $Z$ such that $Z$ and $Y$ are related by a flop (or by a composition of flops). First, we show that the $K_{Z} < 0$ region only contains extremal rays of the cone $\Curv(Z)$ that are of Type (1) or Type (6) (this is Proposition \ref{prop:eliminate6cases}). Then we prove Theorem \ref{thm:mainTheorem} (2), followed by Theorem \ref{thm:mainTheorem} (1), which takes care of the $K_{Z} < 0$ region (meaning that the pseudoautomorphism group acts on all codimension 1 faces of $\Mov(Y)$ corresponding to extremal rays of $\Curv(Z)$ for some $Z$ in the region $K_{Z} < 0$ with finitely many orbits). Finally, we explain how faces corresponding to extremal rays in $-K_{Z}^{\bot}$ are taken care of by results of Kawamata (Theorem \ref{thm:KawamataRelative}). Because no part of $\Curv(Z)$ lies in the $K_{Z} > 0$ region, this completes the proof of Theorem \ref{thm:mainTheorem}.

\begin{proof}
(Proof of Proposition \ref{prop:eliminate6cases}). Let $f: Y \rightarrow X$ be the contraction of a divisor $E \subset Y$ to a point $p \in X$ (so this covers cases (2), (3), (4), and (5)). If we have any curve $C \subset E$, then by definition of the contraction $\mathrm{cont}_{R}$ of an extremal ray $R \subset \Curv(Y)$, the class $[C] \in R \subset \Curv(Y)$. Since $R$ is in the $-K_{Y} > 0$ region, $-K_{Y} \cdot C = f^{\ast}\mathcal{O}(1) \cdot C > 0$. So the induced map $C \rightarrow \mathbb{P}^{1}$ is nonconstant. Now we have a map $f \vert_{E}: E \rightarrow \mathbb{P}^{1}$ from a surface to a curve that does not contract any curves, a contradiction. Therefore this eliminates cases (2), (3), (4), and (5). We can also eliminate (7) by the same argument (replacing $E$ by a fiber of the del Pezzo fibration $\mathrm{cont}_R$).

(8) We cannot have $Y \rightarrow \{p\}$, a point, and $Y$ is a Fano manifold. This is because $-K_{Y}$ defines a fibration $Y \rightarrow \mathbb{P}^{1}$, so $-K_{Y}$ is not ample. This leaves (1) and (6) as the only two possibilities, completing the proof.
\end{proof}

We prove the second claim of Theorem \ref{thm:mainTheorem}.

\begin{proof} (Proof of (2) of Theorem \ref{thm:mainTheorem}). Suppose that the contraction is the blowup of a smooth curve, which may be considered as the contraction of a ruled surface $E$. Then, we have a $\mathbb{P}^{1}$-bundle $g: E \rightarrow \Gamma$, and we claim that it is a trivial $\mathbb{P}^{1}$-bundle in our case. Let $l$ be a fiber of $g$. Because $l \cdot (-K_{Y}) = 1$, and
\begin{align*}
l \cdot (-K_{Y}) &= l \cdot f^{\ast}\mathcal{O}(1) \text{ \; by our assumptions}\\
&= \deg(f \vert_{l}^{\ast} \mathcal{O}(1)),
\end{align*}
so $f \vert_{l}: l \rightarrow \mathbb{P}^{1}$ is an isomorphism. Then, the map $E \rightarrow \Gamma \times \mathbb{P}^{1}$ given by $(g, f \vert_{E})$ is an isomorphism, so that $g: E \rightarrow \Gamma$ is a trivial $\mathbb{P}^{1}$-bundle.

It follows from above that $\Gamma$ is contained in each general fiber $F$ of the {\it K}3 fibration $f$. Now we assume that $H^{3}(Y, \mathbb{C}) = 0$, and we show that this implies that $\Gamma \cong \mathbb{P}^{1}$. If $Y \rightarrow Y^{\prime}$ is the blowup of $\Gamma \subset Y^{\prime}$, then $H^{1}(\Gamma)$ injects into $H^{3}(Y, \mathbb{C})$. Since $H^{3}(Y, \mathbb{C}) = 0$, we have $H^{1}(\Gamma) = 0$, and $\Gamma \cong \mathbb{P}^{1}$. Using the facts that $\Gamma \cong \mathbb{P}^{1}$ and $F$ is a {\it K}3 surface, the adjunction formula shows that $\Gamma$ is a $(-2)$-curve. Therefore every general fiber of the {\it K}3-fibration contains a $(-2)$-curve.

Theorem \ref{thm:SterkKawamata} below (due to Sterk and Kawamata) asserts that the automorphism group of the generic fiber $Y_{\eta}$ of the {\it K}3 fibration $f: Y \rightarrow \mathbb{P}^{1}$ acts on the set of all $(-2)$-curves in $Y_{\eta}$ with finitely many orbits. Because $\Aut(Y_{\eta}) \subset \PsAut(Y)$ (see Proposition \ref{prop:AutSubsetPsAut} below), the larger group $\PsAut(Y)$ must also act on the set of all $(-2)$-curves in $Y_{\eta}$ with finitely many orbits. There is an injection
\begin{center}
$\{E \subset Y \; \vert \; E \text{ an exceptional divisor of Type (1) on }Y\} \hookrightarrow \{(-2)\text{-curves in } Y_{\eta}\}$
\end{center}
which is given by sending any exceptional divisor $E$ of Type (1) to $E_{\eta}$ (the generic fiber of the restriction $f \vert_{E}: E \rightarrow \mathbb{P}^{1}$), and therefore $\PsAut(Y)$ acts on the extremal rays of Type (1) with finitely many orbits, proving (2) of Theorem \ref{thm:mainTheorem}.
\end{proof}

Before proving the first claim of Theorem \ref{thm:mainTheorem}, we first list some important results used in this part of the proof.

\begin{theorem}
\label{thm:SterkKawamata}
(\cite{S85}, Proposition 2.5 and \cite{K97} Remark 2.2 (2)). \label{thm:MCCK3overChar0}
Let $S$ be a {\it K}3 surface over a field $k$ of characteristic zero (not necessarily algebraically closed). Then $\Aut(S)$ acts on $\Nefe(S)$ with a rational polyhedral fundamental domain, and in particular, $\Aut(S)$ acts with finitely many orbits on the set of all $(-2)$-curves $C$ in $S$.
\end{theorem}

Theorem \ref{thm:MCCK3overChar0} follows from Morrison's cone conjecture for {\it K}3 surfaces over $\mathbb{C}$ proved by Sterk (\cite{S85}, Proposition 2.5) together with a generalization by Oguiso-Sakurai in \cite{OS01} and Kawamata in \cite{K97}.

\begin{proposition}
\label{prop:AutSubsetPsAut}
Let $Y$ be a smooth projective 3-fold with a {\it K}3-fibration $f: Y \rightarrow \mathbb{P}^{1}$ where $-K_{Y} = f^{\ast}\mathcal{O}(1)$. Let $Y_{\eta}$ denote the generic fiber of the fibration. Then $\Aut(Y_{\eta}) \subset {\it PsAut(Y)}$.
\end{proposition}

\begin{proof}
(Proof of Proposition \ref{prop:AutSubsetPsAut}). This follows from the following two facts:
\begin{enumerate}
\item $Y \rightarrow \mathbb{P}^{1}$ is a relative minimal model, since $Y$ is smooth and $K_{Y}$ is nef (in fact, trivial) on the fibers of $f$.
\item Let $Y \rightarrow S$ and $Z \rightarrow S$ be relative minimal models and $\alpha: Y \dashrightarrow Z$ a birational map over $S$. Then $\alpha$ is an isomorphism in codimension 1 (\cite{K97}, Theorem 3.52 (2)).
\end{enumerate}
\end{proof}

Another result of Sterk that we need is a consequence of the Morrison cone conjecture for {\it K}3 surfaces.

\begin{theorem}
(Sterk, \cite{S85}, Proposition 2.6). \label{thm:SterkLineBundle}
Let $F$ be a {\it K}3 surface. Consider line bundles $L$ such that $L$ is nef and $L^{2} = 2k$ for some fixed $k \in \mathbb{N}$. Then $\Aut(F)$ acts on the set of all such $L$ with finitely many orbits.
\end{theorem}

Now we prove the first claim of Theorem \ref{thm:mainTheorem}. \\

\begin{proof} (Proof of (1) of Theorem \ref{thm:mainTheorem}). Suppose that the contraction is a conic bundle $g: Z \rightarrow S$, where $Z$ is a 3-fold. We note that $Z$ may not be the 3-fold $Y$ that we start with, but $Y$ and $Z$ are related by flops $Y \dashrightarrow Z$. We continue to denote the {\it K}3 fibration by $f: Z \rightarrow \mathbb{P}^{1}$.

\begin{center}
\begin{tikzcd}
Y \arrow[r, dotted, "flops"] \arrow[d, phantom]
& Z \arrow[d, "f"] \arrow[r, "g"] & S \\
 \arrow[r, phantom]
& \mathbb{P}^{1}
\end{tikzcd}
\end{center}

Let $F = Z_{\eta}$ be the generic fiber of the {\it K}3 fibration $f$. Denote the extension of scalars $S \times \eta = S \otimes_{\mathbb{C}} \mathbb{C}(\mathbb{P}^{1})$ again by $S$. Define $h:= g \vert_{F}: F \rightarrow S$, a generically finite morphism of degree 2. We prove that $h$ is finite (meaning that no curves $\Gamma \subset F$ are contracted by $h$) by contradiction. Suppose the claim is false, that is, suppose that $h(\Gamma)$ is a point for some $\Gamma \subset F$. Then since $h:=g \vert_{F}$, we know that $g(\Gamma)$ is a point. Because $g$ is the contraction of an extremal ray in the region $K_{Z} < 0$, we have $K_{Z} \cdot \Gamma < 0$. But $K_{Z} \cdot \Gamma = f^{\ast} \mathcal{O}(1) \cdot \Gamma = 0$. Therefore $h$ is finite of degree two.

There is an involution $i$ acting on $F$, associated to $h: F \rightarrow S$. Because $h$ is a degree two cover, it  is given by a quotient by a $\mathbb{Z}/2\mathbb{Z}$-action. Let $i: F \rightarrow F$ be the involution corresponding to $h$.

Since the conic bundle $g$ maps the smooth 3-fold $Y$ to a base $S$, it follows (\cite{M82}) that $S$ is also smooth. By the classification of involutions of {\it K}3 surfaces, there are two possibilities for the fixed locus $\Fix(i)$ of the involution $i$: either $\Fix(i)$ is empty, or $\Fix(i)$ is the union of smooth curves. (Note that there also exist examples of involutions of {\it K}3 surfaces where $\Fix(i)$ has isolated fixed points. But then $S$ would be singular, and therefore this is not a possibility in our setting.) In the first case, the base $S$ must be an Enriques surface, and in the second case $S$ is a rational surface.

If $S$ is an Enriques surface, then we can explicitly describe the movable cone of $Y$ and show that the cone conjecture holds in this case (see Example \ref{ex:Enriques}).

This leaves the second case, where $S$ is a rational surface. First, Riemann-Hurwitz gives $K_{F} = h^{\ast}(K_{S} + \displaystyle{\frac{1}{2} B})$. Since $F$ is {\it K}3, it follows that $K_{F} \equiv 0$, so $K_{S} + \displaystyle{\frac{1}{2}}B \equiv 0$, and thus $B \in \vert -2K_{S} \vert$. Next, run MMP on $S$ to obtain $\pi: S \rightarrow \bar{S}$, where $\bar{S} = \mathbb{P}^{2}$ or $\mathbb{F}_{n}$ with $n \neq 1$. We claim that $n \leq 4$. The idea is that when we have a {\it K}3 cover, the complete linear system $\vert -2K_{\bar{S}} \vert$ does not have a multiple fixed component, and thus $n \leq 4$. 

We show this in more detail here. Suppose that $n > 4$. The diagram below shows the relevant maps.
\begin{center}
\begin{tikzcd}
F \arrow[r] \arrow[d, "h"]
& \bar{F} \arrow[d, "\bar{h}"] \\
S \arrow[r, "MMP"] \arrow[d, hookleftarrow] 
& \bar{S} \arrow[d, hookleftarrow] \\
B \arrow[r]
& \bar{B}
\end{tikzcd}
\end{center}
We write $\bar{B} = Im(B) \subset \bar{S}$. On $\bar{S} = \mathbb{F}_{n}$, let $f$ denote a fiber and $C_{1}$ the positive section and $C_{0}$ the negative section. Observe that $C_{1} = C_{0} + nf$. We show that $\vert -2K_{\bar{S}} \vert$ contains $2C_{0}$. Since $-K_{\mathbb{F}_{n}} = D$, the toric boundary, we can write
\begin{align*}
-K_{\mathbb{F}_{n}} &= C_{0} + f + C_{1} + f \\
&= 2C_{0} + (n+2)f,
\end{align*}
using the observation above. Then,
\begin{center}
$-K_{\mathbb{F}_{n}} \cdot C_{0} = -(n - 2) < 0 \text{ for } n>2$.
\end{center}
Thus $C_{0}$ is fixed in $\vert -2K_{\mathbb{F}_{n}} \vert$ for $n > 4$ and $C_{0} \subset \bar{B}$. On the other hand,
\begin{align*}
(\vert -2K_{\mathbb{F}_{n}} - C_{0} \vert) \cdot C_{0} &= -2(K_{\mathbb{F}_{n}} \cdot C_{0}) - C_{0}^{2} \\
&= -2(n-2) + n \\
&= -2n + 4 + n \\
&-(n - 4) \\
&< 0 \text{ for } n>4.
\end{align*}
Then for $n > 4$, the curve $C_{0}$ is fixed in $\vert -2K_{\mathbb{F}_{n}} - C_{0} \vert$. Therefore $2C_{0} \subset \bar{B}$. But this is a contradiction because we see that the double cover is singular along the curve $\bar{h}^{-1}(C_{0})$, but the cover should be smooth there. Therefore $n \leq 4$.

Next, we construct a birational morphism $\phi: \bar{S} \rightarrow S$ by taking $\phi$ to be the identity if $\bar{S} = \mathbb{P}^{2}$, or in the latter case (that is if $\bar{S} = \mathbb{F}_{n}$ with $0 \leq n \leq 4$ and $n \neq 1$), by letting $\phi: \mathbb{F}_{n} \rightarrow \mathbb{P}(1, 1, n)$ be the contraction of the negative section of $\mathbb{F}_{n}$. Now we have the sequence of maps
\begin{center}
$F \xrightarrow{h} S \xrightarrow{\pi} \bar{S} \xrightarrow{\phi} \hat{S}$.
\end{center}
Let $\theta = (\pi \circ h): F \rightarrow \bar{S}$. If $\bar{S} = \mathbb{P}^{2}$, take $L = \theta^{\ast}M$, where $M = H$ is a hyperplane class. Then $M^{2} = 1$ and $L^{2} = 2$. Suppose that $\bar{S} = \mathbb{F}_{n}$ (with $0 \leq n \leq 4$ and $n \neq 1$). Take $L = \theta^{\ast}M$, where $M$ is the positive section of $\mathbb{F}_{n}$. Then $M^{2} = n$ and therefore $L^{2} = 2n \leq 8$ (because $n \leq 4$). This proves that in any case (if $S = \mathbb{P}^{2}$ or $\mathbb{F}_{n}$, where $n \leq 4$ and $n \neq 1$), there is a line bundle $L$ where $L^{2}$ is bounded, and thus we may apply Theorem \ref{thm:SterkLineBundle} to conclude that $\Aut(F)$ acts with finitely many orbits on such $L$.

Now we have $\theta: F \rightarrow \bar{S}$ where $\theta = \pi \circ h$ and $L = \theta^{\ast} M$. We want to show that we can recover the 2-to-1 map $F \rightarrow S$ from $L$. It suffices to show that $L$ defines the morphism $F \rightarrow \hat{S}$. First of all, there is always an injection
\begin{center}
$H^{0}(F, L) \hookleftarrow H^{0}(\hat{S}, M)$,
\end{center}
given by $\theta^{\ast}$. We need to show that this map is also surjective. To do this, it suffices to show that the dimensions of the left and right sides are equal, that is,
\begin{center}
$\dim H^{0}(F, L) = \dim(H^{0}(\hat{S}, M))$.
\end{center}
Recall that $L$ and $M$ are nef and big. Moreover, recall that $F$ is smooth and $\bar{S}$ is log terminal. First of all,
\begin{align*}
h^{0}(L) &= \chi(F, L) \text{ by Kawamata Viehweg vanishing} \\
&= \chi(\mathcal{O}) + \displaystyle{\frac{1}{2}} L \cdot (L - K_{F}) \text{ by Riemann-Roch} \\
&= 2 + \displaystyle{\frac{1}{2}} L^{2} \\
&= 2 + \displaystyle{\frac{1}{2}} (2 \cdot M^{2}) \\
&= 2 + M^{2} \\
&= 2 + n,
\end{align*}
where $\bar{S} = \mathbb{P}(1, 1, n)$ for $n \in \mathbb{N}$. Then on $\bar{S}$, we can compute directly that $h^{0}(M) = n + 2$. Therefore $\dim H^{0}(\bar{D}, L) = \dim(H^{0}(\bar{S}, M))$.
\end{proof}

What remains is to prove Theorem \ref{thm:mainTheorem} for the extremal rays on $-K_{Z}^{\bot}$. We use the following theorem of Kawamata:

\begin{theorem} (Kawamata, \cite{K97}).
\label{thm:KawamataRelative}
Let $Y$ be a 3-fold and $f: Y \rightarrow S$ a {\it K}3 fibration. Then $\PsAut(Y/S)$ acts on $\Mov^{e}(Y/S)$ with finitely many orbits on faces.
\end{theorem}

Kawamata's cone lives in $N^{1}(Y/S) = N^{1}(Y) \times \mathbb{R} / f^{\ast} N^{1}(S) \otimes \mathbb{R}$, which for us is 
\begin{center}
$N^{1}(Y) \otimes \mathbb{R} / \mathbb{R} \cdot [-K_{Y}]$
\end{center}
(this is because in our setting, $S = \mathbb{P}^{1}$ so that $N^{1}(S) = \mathbb{Z} \cdot [-K_{Y}]$). We know that $\PsAut(Y/\mathbb{P}^{1})$ acts on faces of the chamber decomposition containing $[-K_{Y}]$ with finitely many orbits. In our picture, every chamber $\Nef(Z)$ contains $[-K_{Y}]$ because every model has this fibration (the models are related by flops). Then $\PsAut(Y)$ acts on the codimension 1 faces of $Mov(Y)$ with finitely many orbits, taking care of the extremal rays on $-K_{Z}^{\bot}$ and completing the proof of Theorem \ref{thm:mainTheorem}.

Theorem \ref{thm:PsAutActsFiniteOrbitsIfH3zero} follows from Proposition \ref{prop:eliminate6cases}, Theorem \ref{thm:mainTheorem}, and Theorem \ref{thm:KawamataRelative}.

%
%

\section{Examples}
\label{sec:examples}

Let $S$ be a {\it K}3 surface over a field $K$ of characteristic zero. Assume that $K$ is the function field of a complex curve. Then the Brauer group of $K$ is trivial (Tsen's theorem). Then, writing $G = \Gal(\bar{K}/K)$ for the absolute Galois group of $K$, $\Pic(S) = \Pic(S \otimes_{K} \bar{K})^{G}$, see \cite{H16}, Chapter 17 \S 2.2. Assume in addition that $G$ acts trivially on $\Pic(S \otimes_{K} \bar{K})$, so that $\Pic(S) = \Pic(S \otimes_{K} \bar{K})$. Let $W \subset \Aut(\Pic(S))$ be the Weyl group generated by the reflections

\begin{center}
$s_{\alpha}(x) = x + (x.\alpha) \alpha$
\end{center}
for $\alpha \in \Pic(S)$ such that $\alpha^{2} = -2$. Then $\Aut(S)$ is commensurable with $\Aut(\Pic(S))/W$. This follows from the Torelli theorem over $K$ (Theorem \ref{thm:SterkKawamata}), cf. \cite{H16}, Chapter 15 Corollary 2.7. In particular:

\begin{proposition}
\label{prop:AutSfinite2reflexive}
With notation as above, $\Aut(S)$ is finite if and only if $W \subset \Aut(\Pic(S))$ has finite index. Then the lattice $\Pic(S)$ is called 2-reflective in \cite{D83}. \qed
\end{proposition}

\begin{remark}
If $G$ does not act trivially on $\Pic(S \otimes_{K} \bar{K})$, then one can define a reflection group $W^{\prime}$ such that $\Aut(S)$ is commensurable with $\Aut(\Pic(S))/W^{\prime}$ (cf. \cite{OS01}, Definition 1.6), but $W^{\prime} \nsubseteq W$ in general.
\end{remark}

Many examples are obtained via mirror symmetry for Fano 3-folds. Let $X$ be a Fano 3-fold and $E \subset X$ a smooth divisor such that $K_{X} + E = 0$. Then there exists an associated smooth projective 3-fold $Y$ and a {\it K}3-fibration $f: Y \rightarrow \mathbb{P}^{1}$ such that $-K_{Y} = f^{\ast} \mathcal{O}(1)$ and $H^{3}(Y, \mathbb{Z}) = 0$ and there is a fiber $D$ of $f$ which is a normal crossing divisor with a 0-dimensional stratum (\cite{CCGK16}, \cite{CP18}, \cite{P18} \cite{DHKOP23}). In particular, $f: Y \rightarrow \mathbb{P}^{1}$ satisfies the hypotheses of Theorem \ref{thm:PsAutActsFiniteOrbitsIfH3zero}.

\begin{remark}
The Fano 3-fold $X$ is mirror to the {\it K}3-fibration $f: Y \setminus D \rightarrow \mathbb{A}^{1}$ (a so-called Landau-Ginzburg model).
\end{remark}

\noindent Moreover, the Picard lattice $\Pic(Y_{\bar{\eta}})$ of the geometric generic fiber $Y_{\bar{\eta}} = Y_{\eta} \otimes_{\mathbb{C}(\mathbb{P}^{1})} \overline{\mathbb{C}(\mathbb{P}^{1})}$ of $f: Y \rightarrow \mathbb{P}^{1}$ and the action of the Galois group on it were computed (\cite{P18}, \cite{DHKOP23}). There are 105 deformation types of Fano 3-folds. We will use the notation of Mori-Mukai (as in \cite{CCGK16}) for the Fano 3-folds of Picard rank at least 2. The Galois group action on $\Pic(Y_{\bar{\eta}})$ is trivial except in cases 2.12, 4.3, 6.1, 7.1, 8.1, 9.1, and 10.1. In the remaining cases, the lattice $\Pic(Y_{\eta}) = \Pic(Y_{\bar{\eta}})$ is not 2-reflective except in cases 2.1, 3.2, 3.6, 4.9, and the Picard rank 1 cases $X_{6} \subset \mathbb{P}(1, 1, 1, 1, 3)$ and $X_{6} \subset \mathbb{P}(1, 1, 1, 2, 3)$. Here we use Nikulin's classification of 2-reflective lattices (cf. \cite{D83}, Theorem 2.2.2). In the remaining $105 - 13 = 92$ cases, $\Aut(Y_{\eta})$ is infinite by Proposition \ref{prop:AutSfinite2reflexive}, so $\PsAut(Y)$ is infinite (and so $\Mov(Y)$ is not rational polyhedral).

\begin{remark}
The {\it K}3 fibration $f: (Y, D) \rightarrow (\mathbb{P}^{1}, \infty)$ mirror to a Fano 3-fold $X$ deforms in a family of dimension $\rho(X) - 1$ with the same properties (\cite{DHKOP23}, Remark 1.8).
\end{remark}

\begin{example} (Mirror of quartic 3-fold $X_{4} \subset \mathbb{P}^{4}$.)
\label{ex:mirrorQuartic3fold}
Consider a pencil of quartic {\it K}3 surfaces in $\mathbb{P}^{3}$:
\begin{center}
$((X + Y + Z + T)^{4} + t(XYZT) = 0) \subset \mathbb{P}^{3}$, where $t \in \mathbb{P}^{1} = \mathbb{C} \cup \{ \infty \}$.
\end{center}
Let $Y \rightarrow \mathbb{P}^{3}$ be a crepant resolution $Y \rightarrow \hat{Y}$ of the blowup $\hat{Y} \rightarrow \mathbb{P}^{3}$ of the base locus of the pencil. Then $f$ is a {\it K}3 fibration, $-K_{Y} = f^{\ast}\mathcal{O}(1)$, $D = f^{\ast} \infty$ is a type III {\it K}3, and
\begin{center}
$\Pic(Y_{\eta}) \cong \langle -4 \rangle \oplus U \oplus (-E_{8})^{\oplus 2}$
\end{center}
is a non 2-reflective lattice with trivial Galois action. So $\PsAut(Y)$ is infinite.
\end{example}

\begin{example}
\label{ex:Enriques}
Here we describe all examples of Type (6) with base $S$ an Enriques surface. Recall that in this case, we have a conic bundle $g: Y \rightarrow S$. The finite 2-to-1 morphism $(g, f): Y \rightarrow S \times \mathbb{P}^{1}$ has branch locus $\mathcal{B} \sim \mathrm{pr}^{\ast}_{1}(-2K_{S}) + \mathrm{pr}^{\ast}_{2} \mathcal{O}(2)$ by Riemann-Hurwitz and our assumption that $-K_{Y} = f^{\ast} \mathcal{O}(1)$. Since $S$ is an Enriques surface, we have $2K_{S} \sim 0$ and $\mathcal{B} \sim \mathrm{pr}^{\ast}_{2} \mathcal{O}(2)$. So the double cover $(g, f)$ is branched over two fibers of the trivial family $S \times \mathbb{P}^{1} \rightarrow \mathbb{P}^{1}$ and restricts to the universal cover $F \rightarrow S$ of $S$ on the remaining fibers. It follows that
\begin{center}
$Y = (F \times \mathbb{P}^{1})/(\mathbb{Z}/2\mathbb{Z})$, 
\end{center}
where the $\mathbb{Z}/2\mathbb{Z}$ action is given by the involution $i \times j$, where $i$ is the fixed point free involution on $F$ with quotient $S$ and $j: \mathbb{P}^{1} \rightarrow \mathbb{P}^{1}$ with $j(t) = -t$. We consider the conic bundle $g: Y \rightarrow S = F / \langle i \rangle$ as the first projection map and the {\it K}3 fibration $f: Y \rightarrow \mathbb{P}^{1} / \langle j \rangle \cong \mathbb{P}^{1}$ as the second projection.

In general if $G$ is a finite group acting on $X$ and $Y = X / G$, then the nef cone of $Y$ may be described as 
\begin{center}
$\Nef(Y) = \Nef(X) \cap N^{1}(Y)$,
\end{center}
where $N^{1}(Y) = N^{1}(X)^{G} \subset N^{1}(X)$. Therefore, to understand $\Nef(Y)$, it is enough to understand $\Nef(X)$.

In our case, we have $X = F \times \mathbb{P}^{1}$ and $G = \mathbb{Z}/2\mathbb{Z}$. So
\begin{center}
$\Nef(X) = \langle \mathrm{pr}^{\ast}_{1} \Nef(F), \mathrm{pr}^{\ast}_{2} \Nef \mathcal{O}(1) \rangle_{\mathbb{R} \geq 0}$.
\end{center}
Now,
\begin{align*}
\Nef(Y) &= \Nef((F \times \mathbb{P}^{1}) / (\mathbb{Z}/2\mathbb{Z})) \\
&= \Nef(F \times \mathbb{P}^{1}) \cap N^{1}(F \times \mathbb{P}^{1})^{G} \\
&= \langle \mathrm{pr}^{\ast}_{1}(\Nef(F)), \mathrm{pr}^{\ast}_{2}(\mathcal{O}(1)) \rangle \cap N^{1}(F \times \mathbb{P}^{1})^{G} \\
&= \langle g^{\ast}\Nef(S), f^{\ast}\mathcal{O}(1) \rangle_{\mathbb{R} \geq 0}.
\end{align*}

In particular, there are no faces of $\Nef(Y)$ corresponding to small contractions, so $\Mov(Y) = \Nef(Y)$. Finally, $\PsAut(Y) = \Aut(Y)$ acts on $N^{1}(Y) = N^{1}(S) \times N^{1}(\mathbb{P}^{1})$ via the surjection $\Aut(Y) \twoheadrightarrow \Aut(S)$. Now the cone conjecture for Enriques surfaces (see \cite{K97}, Theorem 2.1) implies the cone conjecture for $Y$. In particular, Theorem 1.8 holds.
\end{example}

\end{document}